\documentclass[10pt,twoside]{amsart}

\usepackage{amsmath,amssymb
}
\usepackage{amsfonts}
\usepackage{amsthm,amscd}
\usepackage{amsmath}
\usepackage{amssymb}
\usepackage{amstext}
\usepackage{color}
\usepackage{latexsym}
\usepackage{amscd,graphics}

\begin{document}
\newcommand{\comments}[1]{\marginpar{\footnotesize #1}} 

\newtheorem{proposition}{Proposition}[section]
\newtheorem{lemma}[proposition]{Lemma}
\newtheorem{sublemma}[proposition]{Sublemma}
\newtheorem{theorem}[proposition]{Theorem}

\newtheorem{maintheorem}{Main Theorem}
\newtheorem{corollary}[proposition]{Corollary}

\newtheorem{ex}[proposition]{Example}

\theoremstyle{remark}

\newtheorem{remark}[proposition]{Remark}

\theoremstyle{definition}
\newtheorem{definition}[proposition]{Definition}
\newcommand{\ovB}{\bar{B}}
\def\Erg{\mathrm{Erg}\, }
\def\cone{\mathbf{C}} 
\def\real{\mathbb{R}}
\def\sphere{\mathbf{S}^{d-1}}
\def\integer{\mathbb{Z}}
\def\complex{\mathbb{C}}
\def\BBB{\mathbb{B}}
\def\supp{\mathrm{supp}}
\def\var{\mathrm{var}}
\def\sgn{\mathrm{sgn}}
\def\sp{\mathrm{sp}}
\def\id{\mathrm{id}}
\def\Imm{\mathrm{Image}}
\def\cc{\Subset}
\def\D{\mathrm {d}}
\def\I{i}
\def\E{e}
\def\Lip{\mathrm{Lip}}
\def\BB{\mathcal{B}}
\def\CC{\mathcal{C}}
\def\DD{\mathcal{D}}
\def\EE{\mathcal{E}}
\def\FF{\mathcal{F}}
\def\GG{\mathcal{G}}
\def\II{\mathcal{I}}
\def\JJ{\mathcal{J}}
\def\KK{\mathcal{K}}
\def\LL{\mathcal{L}}
\def\LLL{\mathbb{L}}
\def\MM{\mathcal{M}}
\def\NN{\mathcal{N}}
\def \OO{\mathcal {O}}
\def \PP{\mathcal {P}}
\def \QQ{\mathcal {Q}}
\def \RR{\mathcal {R}}
\def\SS{\mathcal{S}}
\def\TT{\mathcal{T}}
\def\UU{\mathcal{U}}
\def\VV{\mathcal{V}}
\def\YY{\mathcal{Y}}
\def\ZZ{\mathcal{Z}}
\def\FFF{\mathbb{F}}
\def\PPP{\mathbb{P}}

\title[Characteristic functions as bounded multipliers]{Characteristic functions as bounded multipliers on anisotropic spaces}

\author{Viviane Baladi}
\address{Sorbonne Universit\'e, CNRS, IMJ-PRG,  4, Place Jussieu, 75005 Paris, France}
\email{viviane.baladi@imj-prg.fr}

\date{Revised version, January 10, 2018} 
\begin{abstract}
We   show that characteristic functions of domains with piecewise $C^3$
boundaries  transversal to suitable cones are bounded multipliers on 
a recently introduced   scale $\UU^{\cone, t,s}_p$ 
of anisotropic Banach spaces, 
under the  conditions $-1+1/p<s<-t<0$, with $p\in (1,\infty)$.
\end{abstract}
\thanks{I thank D. Terhesiu for many useful comments and encouragements, in particular
during her three-months stay in Paris in 2016. I am very grateful to the anonymous referee for thoughtful
remarks, including the observation that the norms only depend on the ``unstable'' cones $\cone_+$.
I thank Malo J\'ez\'equel for very sharp questions which helped me to improve the text.}


\maketitle

\section{Introduction}

A (not necessarily smooth) function $g:M\to \complex$ is called a bounded multiplier on a Banach space $\BB$ of distributions
on a $d$-dimensional Riemann manifold $M$ if there exists  $C_g<\infty$ so that for all
$\varphi \in \BB$ the product $g\varphi$ is a well-defined element of
$\BB$ and, in addition,   $\|g \cdot \varphi\|\le C_g \|\varphi\|$, where $\|\cdot\|$ is the norm of $\BB$. One interesting special case
is when $g$ is the characteristic function $1_\Lambda$ of an open domain
 $\Lambda\subset M$: Half a century ago, Strichartz \cite{Str} proved
that for any $d\ge 1$,
if $M=\real^d$ and $\BB$ is the Sobolev\footnote{Recall that
$\|\varphi\|_{H^t_p}= \|(\id+\Delta)^{t/2} \varphi\|_{L_p}=\|\FFF^{-1} (1+|\xi|^2)^{t/2}
\FFF \varphi \|_{L_p}$, with $\Delta$ the Laplacian
and $\FFF$ the Fourier transform.} space
$H^t_p(\real^d)$ for $p\in (1,\infty)$ and $t\in \real$,  then the characteristic function
$1_\Lambda$ of a half-space is a bounded multiplier on $H^t_p(\real^d)$ if and only if $-1+1/p<t<1/p$.

In the present work, we consider a newly introduced scale $\UU^{\cone, t,s}_p$ of
spaces  of anisotropic distributions $\BB$ on a  manifold $M$, adapted to  smooth hyperbolic dynamics, and we prove the bounded multiplier 
property for characteristic functions of suitable subsets $\Lambda\subset M$.

Fix $r>1$, and suppose from now on that $M$ is connected and compact.
The simplest hyperbolic maps on $M$ are transitive $C^r$ Anosov diffeomorphisms $T$.
The Ruelle transfer operator
associated to  such a map $T$ and to a  $C^{r-1}$ function $h$ on $M$ 
(for example, $h=1/|\det DT|$) is defined on
$C^{r-1}$ functions $\varphi$ by
\begin{equation}\label{leqa}
\LL_h \varphi= (h \cdot \varphi ) \circ T^{-1} \, .
\end{equation}
Blank--Keller--Liverani \cite{BKL} were the first to study
the spectrum of such transfer operators on a suitable Banach space $\BB$ of {\it anisotropic  distributions}  and to exploit this spectrum to get information on
the Sinai--Ruelle--Bowen (physical) measure: The spectral radius 
of $\LL_{1/|\det DT|}$ is equal to $1$, there is a simple positive
maximal eigenvalue, whose eigenvector is in fact a Radon measure
$\mu$, which is just the physical measure  of $T$. Finally,
the rest of the spectrum lies in a disc of radius strictly
smaller than $1$,
  which implies exponential decay of correlations
$\int \varphi (\psi \circ T^{n}) \D \mu-\int \varphi \D \mu \int \psi \D \mu$
for H\"older observables $\psi$ and $\varphi$ as $n\to \infty$. (The first step 
in this analysis is to show the bound $\rho_{ess}<1$ for the
essential spectral radius of $\LL_{1/|\det DT|}$ on $\BB$.)

Some natural  dynamical systems originating from physics (such as Sinai billiards) enjoy
uniform hyperbolicity, but are only {\it piecewise smooth.} Letting $M=\cup_i \Lambda_i$ be
a (finite or countable) partition of $M$ into domains where the dynamics is smooth, one 
can often reduce to the
 smooth hyperbolic case via the decomposition
\begin{equation}\label{leqapw}
\LL_{1/|\det DT|} \varphi= \sum_i  \frac{ (1_{\Lambda_i}\cdot \varphi) }{|\det DT|}
 \circ T^{-1} \, .
\end{equation}
This  motivates studying bounded
multiplier properties of characteristic functions.

\smallskip
 
In the 15 years since the publication of \cite{BKL}, 
dynamicists and semi-classical analysts have
created a rich jungle of spaces of anisotropic distributions for hyperbolic
dynamics (here, $d=d_s+d_u$ with $d_s\ge 1$ and
$d_u\ge 1$). These spaces are usually
scaled by two real numbers $v<0$ and $t>0$. 
Leaving aside  
the 
classical foliated anisotropic spaces of Triebel \cite{Tr} (which are limited
to ``bunched'' cases \cite{BG2}, and seem to fail for Sinai
billiards), they come in two groups: 

In the first, ``geometric''
group \cite{BKL, GL1},  a class of $d_s$-dimensional ``admissible''
leaves $\Gamma$ (having tangent vectors in stable cones for $T$)
is introduced, and the norm of $\varphi$  is obtained by fixing
an integer $t\ge 1$ and taking a supremum, over
all admissible leaves $\Gamma$, of the 
  partial derivatives of $\varphi$ of total order  at most $t$, integrated  against
$C^{|v|}$ test functions on $\Gamma$.
Modifications of this space,  for suitable noninteger  $0<t<1$ and $|v|<1$, were introduced to work with
piecewise smooth systems \cite{DL,DZ}
(only in dimension two). A version of these spaces for piecewise smooth hyperbolic flows
in dimension three recently allowed to prove exponential mixing for Sinai 
billiard flows \cite{BDL}. 

In the\footnote{This group could also be called pseudodifferential, or semi-classical, or Sobolev.} second, ``microlocal,'' group 
\cite{BT1}, a third parameter $p \in [1,\infty)$ is present, and
the norm (in charts) of $\varphi$ is the $L_p$ average   of $\Delta^{t,v}(\varphi)$,
where the operator $\Delta^{t,v}$ interpolates smoothly 
between $(\id +\Delta)^{v/2}$  in {\it stable cones} in the cotangent space, and $(\id +\Delta)^{t/2}$  in {\it unstable cones}
in the cotangent space.
Powerful tools are available for this microlocal  approach, 
allowing in particular to study the dynamical determinants and zeta 
functions\footnote{The ``kneading determinants'' of
by Milnor and Thurston from the 70's are revisited as ``nuclear
decompositions'' in \cite{Ba}.}
much more efficiently than for the geometric spaces. Variants of these microlocal
spaces (usually in the Hilbert setting $p=2$) have also been studied by the semi-classical community,
starting from \cite{FRS}.  However, S. Gou\"ezel pointed
out over ten years ago that {\it characteristic functions cannot be bounded
multipliers} on spaces defined by conical wave front sets
as in \cite{BT1} or \cite{FRS} (Gou\"ezel's counterexamples
are presented in  \cite[App. 1]{Baladijoel}). The microlocal spaces of the type defined in \cite{BT1, BT2} or \cite{FRS}
thus appear {\it unsuitable}
to study piecewise smooth dynamics. 

In order to overcome this limitation of the microlocal
approach, we recently introduced \cite{Baladijoel} a new scale $\UU^{\cone, t,s}_p$ of microlocal anisotropic spaces,
obtained by mimicking the construction of  the  geometric spaces 
of Gou\"ezel--Liverani \cite{GL1} (with, morally, $s=v+t$).
We showed in \cite{Baladijoel} the expected bound on the essential
spectral radius of the transfer operator
of a $C^r$ Anosov diffeomorphism acting on  $\UU^{\cone, t,s}_p$
(if $t-(r-1)<s<-t<0$),
and we conjectured  that characteristic functions
of  domains with piecewise smooth boundaries
everywhere transversal to the stable cones should be bounded multipliers on  $\UU^{\cone, t,s}_p$, if $s$ and $t$
satisfy additional constraints depending on $p\in (0,1)$. {\it The main result\footnote{See Remark~\ref{beware}.} of
the present paper, Theorem~\ref{propp}, implies this bounded multiplier
property if $\max \{t-(r-1), -1+1/p\}<s<-t<0$.}


This result opens the door to the spectral study, not only of
hyperbolic maps with discontinuities in arbitrary dimensions, but also
(using nuclear power decompositions \cite{Ba, Baladijoel}) of the hitherto
unexplored topic of
the dynamical zeta functions of piecewise expanding and piecewise hyperbolic maps
in any dimensions. This should include
billiards maps \cite{DZ} and their dynamical zeta functions in
arbitrary dimensions.
We also hope that the spaces $\UU^{\cone, t,s}_p$ will allow to extend the scope of the renewal
methods introduced in \cite{LT} to  dynamical systems
with  infinite invariant measures. (The induction procedure  used there
introduces discontinuities in the dynamics.)
Finally, it goes without saying that suitable version of the spaces $\UU^{\cone, t,s}_p$ will
be useful to study flows.


F.~Faure and M.~Tsujii 
\cite{FaTsm} recently introduced  new microlocal anisotropic spaces,
for
which the wave front set is more narrowly constrained than for
previous microlocal spaces used for hyperbolic dynamics. It would be interesting to check whether
characteristic functions are bounded multipliers on these new spaces. (Note however
that, contrary to the spaces $\UU^{\cone, t,s}_p$
or the spaces of \cite{FRS, BT1,GL1,DZ},   spaces of
\cite{FaTsm}   do  not appear suitable for perturbations of hyperbolic maps
or flows.)


\section{$\UU^{\cone, t,s}_p$: A Fourier version of the Demers--Gou\"ezel--Liverani spaces}
\label{SOBS}

We recall the ``microlocal'' spaces $\UU^{\cone, t,s}_p$,  for real numbers $s$ and $t$
(in the application, $s<-t<0$) and
$1\le p \le \infty$,
introduced in \cite{Baladijoel}.

\subsection{Basic notation}
\label{motivv}

Suppose that $d=d_s+d_u$ with $d_u\ge 1$ and $d_s\ge 1$. For  
$\ell \ge 1$ and $x\in \real^\ell$,
$\xi \in \real^\ell$,
we write $x\xi$ for the scalar product of $x$ and $\xi$. The Fourier transform
$\FFF$ and its inverse  $\FFF^{-1}$ are defined on 
 rapidly decreasing functions $\varphi, \psi$  by
\begin{align}
\label{Ftran}
\FFF(\varphi)(\xi)&=
\int_{\real^d} \E^{-\I x\xi} \varphi(x)\D x\, , \quad \xi \in \real ^d\, , \\
\label{Ftrani}\FFF^{-1}(\psi)(x)&= \frac{1}{(2\pi)^d}
\int_{\real^d} \E ^{\I x\xi}  \psi(\xi)\D \xi \, , \quad x\in \real ^d \, ,
\end{align}
and extended to the space  of temperate distributions $\varphi, \psi$  as usual
\cite{RS}.
For suitable functions
 $a:\real^{d} \to \real$ (called ``symbols'', note that, in this paper, $a$ depends only on $\xi$,
while more general symbols may depend on $x$ and $\xi$), we define an operator
 $a^{Op}$ acting on suitable 
 $\varphi: \real^d \to \complex$, by
 \begin{equation}\label{neat}
a^{Op} (\varphi)=\FFF^{-1} (a(\cdot)\cdot \FFF(\varphi) )=(\FFF^{-1} a) * \varphi\, .
\end{equation}
Note that
$\|a^{Op} \varphi \|_{L_p} \le \|\FFF^{-1} a\|_1 \|\varphi\|_{L_p}$ for  each $1\le p\le \infty$, by  Young's inequality  in $L_p$.

Fix a $C^\infty$ function $\chi:\real_+\to [0,1]$ with
$\chi(x)=1$ for $x\le 1$, and $\chi(x)=0$ for $x\ge 2$.
For  $D\ge 1$, define $\psi^{(D)}_n:\real^D\to [0,1]$ for $n\in \mathbb Z_+$, by
$\psi^{(D)}_0(\xi)= \chi(\|\xi\|)$, and
\begin{equation}\label{2.49}
\psi^{(D)}_n(\xi)= \chi(2^{-n}\|\xi\|)- \chi(2^{-n+1}\|\xi\|) \, , 
\quad n\ge 1 \,  .
\end{equation}
We set $\psi_n=\psi^{(d)}_n$.
Note that 
$$
\FFF^{-1} \psi^{(D)}_n=2^{D(n-1)} \FFF^{-1} \psi^{(D)}_1(2^{n-1} x)
\mbox{ and } (\sum_{k\le n} \FFF^{-1} \psi^{(D)}_k) (x)=2^{Dn} \FFF^{-1}  \chi(2^n x)
\, ,
$$
 so that, for any $D$,
\begin{equation}\label {forYoung}
\sup_n \|\FFF^{-1} \psi^{(D)}_n\|_{L_1(\real^D)}<\infty\, ,\qquad
\sup_n \|\sum_{k\le n} \FFF^{-1} \psi^{(D)}_k\|_{L_1(\real^D)}<\infty\, ,
\end{equation} 
and for every multi-index $\beta$, there exists a constant $C_\beta$ such that
\begin{equation}\label{betagood}
\|\partial^\beta \psi^{(D)}_n\|_{L_\infty} \le C_\beta 2^{-n|\beta|}\, ,
\qquad  \forall\,  n\ge 0\, .
\end{equation}
We shall work with the following operators $(\psi^{(D)}_{n})^{Op}$ 
(putting $\psi_n^{Op}=(\psi^{(d)}_{n})^{Op}$):
\begin{align*}
(\psi_{n}^{(D)})^{Op}(\varphi)(x)&=
\frac{1}{(2\pi)^{d}}
\int_{y\in \real^{d}}\int_{\eta \in \real^{d}}
\E^{\I (x- y)\eta} 
\psi^{(D)}_{n}(\eta)
\varphi(y)  \D \eta \D y \, .
\end{align*}
Note finally the following almost orthogonality property 
\begin{equation}\label{orthog}
(\psi^{(D)}_{n})^{Op}\circ   (\psi^{(D)}_{m})^{Op} \equiv 0\, \quad\mbox{ if }   |n-m| \ge 2\, .
\end{equation}

\subsection{The local anisotropic spaces $\UU^{\cone_+,t,s}_{p}(K)$ for compact $K\subset \real^d$}

Recall that a cone   is a subset of $\real^d$ invariant under scalar
multiplication.
For two cones $\cone$ and $\cone'$
in $\real^d$, we write 
$\cone \cc \cone'$  
if
$\overline \cone\subset \mbox{ interior} \, (\cone' )\cup \{0\}$.
We say that a cone $\cone$ is $d'$-dimensional 
 if $d'\ge 1$
is the maximal dimension of a linear subset of $\cone$.

\begin{definition}
\label{defpol} An  {\it unstable cone} is  a closed cone $\cone_+$ with nonempty interior 
  of  dimension $d_u$ 
in $\real^d$  so that  $\real^{d_s} \times \{0\}$ is included in\footnote{In Definitions 3.2 and
3.3, and 7 lines above  Definition 3.2 of \cite{Baladijoel},  
the condition ``$\real^{d_s} \times \{0\}$ is included in $\cone_-$'' can be replaced by this condition.}
$(\real^d \setminus \cone_+)\cup \{0\}$.
\end{definition}

\noindent Recall that $r>1$. The next
key ingredient is  adapted from \cite{BT2}:

\begin{definition}[Admissible (or fake) stable leaves]\label{fakke}
Let  $\cone_+$ be an unstable cone, and let $C_\FF > 1$.
Then $\FF(\cone_+,C_\FF,r)$ (or just 
$\FF$) is the set
of all $C^r$ (embedded) submanifolds $\Gamma\subset \real^d$, of dimension $d_s$,
with $C^r$ norms of submanifold
charts  $\le C_\FF$, and
so that the straight line connecting any two distinct points in 
$\Gamma$ is normal to a $d_u$-dimensional subspace contained in  $\cone_+$. 
Denote by $\pi_-$ the orthogonal projection from
$\real^d$ to the quotient $\real^{d_s}$ and by $\pi_\Gamma$ its restriction to $\Gamma$.
Our  assumption   implies that $\pi_\Gamma:\Gamma \to \real^{d_s}$ is a $C^r$ diffeomorphism
onto its image with a $C^r$ inverse, whose $C^r$ norm is bounded by a universal scalar
multiple of $\CC_F$. In the sequel, we replace $\CC_F$ by this larger constant 
and \emph{we restrict  to those $\Gamma$ so  that $\pi_\Gamma$ is surjective.}
\end{definition}

\begin{definition}[Isotropic  norm on stable leaves]\label{iso} 
Fix an unstable cone  $\cone_+$.
Let $\Gamma \in \FF(\cone_+, C_\FF,r)$  and 
  let $\varphi \in C^0(\Gamma)$.
For $w\in \Gamma\subset \real^d$, we set
\begin{align}\label{myop}
\psi_{\ell_s}^{Op(\Gamma)}(\varphi)(w)&=
\frac{1}{(2\pi)^{d_s}}
\int_{z\in \real^{d_s}}\int_{\eta_s \in \real^{d_s}}
\E^{\I (\pi_\Gamma(w)- z)\eta_s} 
\psi_{\ell_s}^{(d_s)}(\eta_s)
\varphi(\pi_{\Gamma}^{-1}(z))  \D \eta_s \D z \, ,
\end{align}
where  $\psi_k^{(d_s)}:\real^{d_s}\to [0,1]$ 
is defined  in \eqref{2.49}. 
For all real numbers $1\le p \le \infty$, 
and $-(r-1)<s<r-1$, define an auxiliary isotropic  norm on $C^0(\Gamma)$   as
\begin{equation} \label{ust}
\|\varphi\|^s_{p,\Gamma}=
\sup_{\ell_s \in \integer_+}
 2^{\ell_s s}\|\psi^{Op(\Gamma)}_{\ell_s}(\varphi)\|_{L_p(\mu_\Gamma)}
\, ,
\end{equation}
where $\mu_\Gamma$ is the Riemann volume on $\Gamma$ induced by the standard metric on $\real^d$.
\end{definition}

Note that \eqref{ust} is equivalent, uniformly in $\Gamma \in \FF$, to 
the  (\cite[\S 2.1, Def. 2]{RS}) classical $d_s$-dimensional Besov norm
$B^s_{p,\infty}$ of $\varphi$ in the chart given by $\pi_\Gamma^{-1}$:
$$
\|\varphi\|^s_{p,\Gamma}\sim 
\|\varphi \circ \pi_\Gamma^{-1}\|_{B^s_{p,\infty}(\real^{d_s})} \, .
$$

We next revisit 
the local space given in \cite{Baladijoel}:

\begin{definition}[The local space $\UU^{\cone_+,t,s}_{p}(K)$]\label{below}
Let $r>1$, let $K\subset \real^d$ be a non-empty compact set. 
For an unstable cone
$\cone_+$, a constant
$C_\FF\ge 1$,
and   real numbers $1 \le p \le \infty$, 
and   $t-(r-1)<s<-t<0$,   define
for $\varphi\in L_\infty$  supported
in $K$,  
\begin{equation}\label{def:normOnRUU}
\|\varphi\|_{\UU^{\cone_+,t,s}_{p}}=
\sup_{\Gamma \in \FF(\cone_+,C_\FF,r)} 
 \sup_{\ell \in \integer_+}
2^{\ell t}   \|  \psi_{\ell}^{Op} (\varphi )\|^s_{p,\Gamma}\, .
\end{equation}
Set $\UU^{t,s}_{p}(K)=\UU^{\cone_+,t,s}_{p}(K)$ to
be the completion of $\{\varphi \in L_{\infty}(K) \mid \|\varphi \|_{\UU^{\cone_+, t,s}_{p}}<\infty\}$  for  
the norm $\|\cdot \|_{\UU^{\cone_+, t,s}_{p}}$.  (Note that $\UU^{t,s}_{p}(K)$ also depends on $r$ and $C_\FF$.)
\end{definition}

\begin{remark}\label{beware} Beware that,  in \cite[Definition 3.3]{Baladijoel}, 
the space  $\UU^{t,s}_{p}(K)$
was defined by  completing  $C^\infty(K)$
(or, equivalently, by  \cite[Lemma 3.4]{Baladijoel} and mollification, $C^{r-1}(K)$).
We do not claim that $C^{\infty}(K)$ is dense in the space $\UU^{t,s}_{p}(K)$ from
Definition~\ref{below}. (See however
\cite[Lemmas 3.7, 3.8]{DZ}.)
But, since all results in \cite{Baladijoel} hold (except the heuristic remark  after \cite[Definition B.1]{Baladijoel}), with the same\footnote{In particular, \cite[Lemma C.1]{Baladijoel} holds
replacing $C^\infty(K)$ by compactly supported distributions.} proofs, for the completion
used in Definition~\ref{below}, we may (abusively) use here the same notation $\UU^{t,s}_{p}(K)$.
The new definition is useful to show that \eqref{BMP} implies that $1_\Lambda \UU^{t,s}_{p}(K)
\subset \UU^{t,s}_{p}(K)$.
\end{remark}


The following lemma was proved\footnote{Injectivity of the embedding
into distributions follows from injectivity of the embedding of the closure of the
(larger) set of those tempered distributions $\varphi$ so that $ \|\varphi \|_{\UU^{t,s}_{p}}<\infty$.}  in \cite{Baladijoel}:
\begin{lemma}[Comparing $\UU_p^{\cone_+,t,s}(K)$ with classical spaces]\label{lm:CsU}
Assume $-(r-1)<s<-t<0$.
For any $u> t$, there exists a constant $C=C(u,K)$ such that $\|\varphi\|_{\UU^{\cone_+,t,s}_p} \le C \|\varphi\|_{C^u}$ for all $\varphi\in C^u(K)$. 
For any $u>|t+s|$, the space $\UU^{\cone_+,t,s}_p(K)$ is contained in the space of distributions of order $u$ supported on $K$. 
\end{lemma}

\subsection{The global spaces $\UU^{\cone, t,s}_{p}$ of anisotropic distributions} 

We finally introduce the global spaces $\UU^{\cone, t,s}_{p}$  of distributions
on a compact manifold $M$.

\begin{definition}\label{ChP}
An {\it admissible chart system and partition of unity}
is
a finite system of  local charts $\{(V_\omega, \kappa_\omega)\}_{\omega\in \Omega}$, with 
open subsets
$V_\omega\subset M$, and $C^\infty$ diffeomorphisms
$\kappa_\omega : U_ \omega\to V_\omega$ such that $M \subset \cup_\omega V_\omega$, and  
$U_\omega\subset \real^d$  
is bounded and open, together
with a   $C^{\infty}$ partition of  unity  $\{\theta_\omega\}_{\omega\in \Omega}$ for $M$,
   subordinate to the cover $\VV=\{V_\omega\}$. 
\end{definition}

\begin{definition}[Anisotropic spaces  $\UU^{\cone, t,s}_{p}$ on $M$]\label{defnormU}
Fix $r>1$, an  admissible chart system and partition of unity,   
$\CC_F\ge 1$ and  a system of  cones
$
\cone=\{\cone_{\omega,+}\}_{\omega\in \Omega}
$.
Fix $1\le p\le \infty$, and  real numbers $-(r-1)<s<-t <0$.
The Banach space $\UU^{\cone,t,s}_p=\UU^{\cone,t,s,r,C_\FF}_p$  is 
the completion (see Remark~\ref{beware}) of  $\{\varphi \in L_{\infty}(M) \mid \|\varphi \|_{\UU^{\cone, t,s}_{p}}<\infty\}$ for the norm
$
\|\varphi\|_{\UU^{\cone, t,s}_{p}}:=\max_{\omega \in \Omega} 
\|(\theta_\omega\cdot \varphi)\circ \kappa_\omega\|_{\UU^{\cone_{\omega,+}, t,s}_{p}}
$.
\end{definition}

\begin{remark}[Admissible systems $\{\cone_{\omega,\pm}\}$]
To get  a spectral gap
for the transfer operator $\LL_{1/|\det DT|}$ associated to  a
$C^{\tilde r}$  Anosov diffeomorphism $T$ for $\tilde r>1$,  one must 
take $r=\tilde r$ and
consider an admissible chart system and partition of unity,  with cones $\{\cone_{\omega,+}\}$, satisfying the following conditions \cite{Baladijoel}: 

a) Let $E^s$ and $E^u$ be the  stable, respectively unstable,
bundles of $T$. Then
if $x\in V_\omega$, the cone
$(D\kappa_\omega^{-1})^{*}_x(\cone_{\omega,+})$    
contains the ($d_u$-dimensional) normal subspace of $E^s(x)$, and there exists a $d_s$-dimensional cone
$\cone_{\omega,-}$,  with nonempty interior, so that $\cone_{\omega,+}\cap \cone_{\omega,-}=\{0\}$, and
so that
$(D\kappa_\omega^{-1})^{*}_x(\cone_{\omega,-})$
contains the ($d_s$-dimensional)
normal subspace of  $E^u(x)$.

b)
If  $V_{\omega'\omega}=T(V_\omega)\cap V_{\omega'}\ne \emptyset$,  the $C^r$ map 
corresponding to $T^{-1}$ in charts, 
\[
F=F_{\omega'\omega}=\kappa^{-1}_\omega
\circ T^{-1}\circ \kappa_{\omega'}
:\kappa_{\omega'}^{-1}(V_{\omega'\omega}) \to U_\omega \, ,
\]
extends to a  bilipschitz $C^1$
diffeomorphism of $\real^d$ so that (by definition, $\cone_{\omega',-}\cc
(\real^{d_s} \setminus \cone_{\omega',+})$)
$$DF_{x}^{tr}(\real^d\setminus \cone_{\omega,+}) \cc \cone_{\omega',-}\, , \qquad 
\forall x\in \real^d\, .
$$

c)
Furthermore, there exists, for each $x,y$,
a linear transformation $\mathbb L_{xy}$ so that
$$(\mathbb L_{xy})^{tr}(\real^d\setminus \cone_{\omega, +}) \cc  \cone_{\omega',-}\, 
\mbox{ and }\, 
\mathbb L_{xy}(x-y)=F(x)- F(y)\, .
$$
 A  
map $F$  satisfying (b--c) is called {\it regular
 cone hyperbolic} from 
$\cone_{\omega,\pm}$ to $\cone_{\omega',\pm}$.


The anisotropic spaces $\UU^{\cone, t,s}_1$ (with $p=1$) are analogues of the Blank--Keller--Gou\"ezel--Liverani \cite{BKL,GL1}
spaces $\BB^{t,|s+t|}$ associated to $T$, for integer $t$ and $s<-t$. 
The spaces $\UU^{\cone, t,s}_p$ are somewhat similar to  the Demers--Liverani spaces  \cite{DL} when $p>1$
and $-1+1/p<s<-t<0$.  See \cite{Baladijoel}.
\end{remark}


\section{Characteristic functions as bounded multipliers}

\subsection{Statement of the main result}

Fix $r>1$, $C_\FF>0$, $p\in(1, \infty)$,  an admissible chart system and partition of unity on $M$ (Definition~\ref{ChP}), and an associated cone system 
$\cone=\{\cone_{\omega,+}\}$.
Let $\tilde \Lambda \subset M$ be an open set so that $\partial \tilde \Lambda$ is a finite union of 
$C^r$ hypersurfaces $\partial \tilde \Lambda_i$ so that the normal
 vector at any $x\in \partial \tilde \Lambda_i\cap V_\omega$ lies in 
$\real^d \setminus \cone_{\omega,+}$ (a  transversality condition). 
We claim that if  $\max\{t-(r-1), -1+1/p\}<s<-t<0$  then, for any\footnote{Given 
two cone systems of same cardinality, $\cone \cc \tilde \cone$ means
$\cone_{\omega,+} \cc \tilde \cone_{\omega,+}$ for all $\omega$.}
cone system $\tilde \cone$ with\footnote{Enlarging the cones is  not a problem
when studying  $1_{\tilde \Lambda} ((f \varphi) \circ F)$ for a $C^{r-1}$ function $f$ and 
a $C^r$ regular cone-hyperbolic map $F$
from $\cone$ to $\tilde \cone$ with $\cone \cc \tilde \cone$,
since  the Lasota--Yorke estimate   \cite[Lemma 4.2]{Baladijoel}
gives   $\|(f \varphi )\circ F \|_{\UU^{\tilde \cone, t,s}_p}\le C_{f,F} \| \varphi \|_{\UU^{ \cone ,t,s}_p}$.} $\cone \cc \tilde \cone$,
there exists $C_{\tilde \Lambda, \tilde \cone}<\infty$ so that
$$\|1_{\tilde \Lambda} \varphi \|_{\UU^{\cone, t,s}_p}\le C_{\tilde \Lambda,\tilde \cone} \| \varphi \|_{\UU^{\tilde \cone ,t,s}_p}
\, , \quad \forall \varphi \, .
$$
Since $t-(r-1)<s<-t$, by using suitable $C^\infty$ partitions of
unity $h_j$ and $C^r$ coordinates $F_j$ (arbitrarily close to
the identity, and thus regular cone hyperbolic from $\tilde \cone$
to $\cone$ if $\cone \cc \tilde \cone$), and exploiting
the  Lasota--Yorke  estimate   \cite[Lemma 4.2]{Baladijoel}  for the corresponding
transfer operators,
we reduce to:

\begin{theorem}[Characteristic functions of half-spaces]\label{propp}
Fix  $r>1$, $C_\FF>0$, and an unstable cone $\cone_+$. Let $K \subset \real^d$ be compact,
 and let $\tilde \Lambda\subset \real^d$ be a half-space
whose unit  normal vector $u_{\tilde \Lambda}$ lies in  $\real^d \setminus \cone_+$.
 Then for any 
$$
1<p<\infty \mbox{ and } \max\{t-(r-1), -1+\frac{1}{p}, \} < s <-t< 0  \, ,
$$
there
exists $C<\infty$ so that for any $\varphi \in \UU^{\cone_+,t,s}_p(K)$ we have, 
\begin{equation}\label{BMP}
\|1_{\tilde \Lambda} \varphi\|_{\UU^{\cone_{+}, t,s}_p}\le C \|\varphi\|_{\UU^{\cone_{+}, t,s}_p} \, .
\end{equation}
\end{theorem}

Since $1_{\tilde \Lambda} \varphi \in L_\infty$ if $\varphi\in L_\infty$ and since $\UU^{\cone_+,t,s}_p(K)$
is the completion of a set of bounded functions, the bound \eqref{BMP}
implies that $1_{\tilde \Lambda} \varphi \in \UU^{\cone_+,t,s}_p(K)$ if  $\varphi \in \UU^{\cone_+,t,s}_p(K)$
(use Cauchy sequences).

The conditions in the theorem imply $t<1-1/p$. (This does not imply $t<1/p$
if $p>2$.)

\begin{remark}[Heuristic proof via interpolation: $t<1/p$ vs. $t<\min\{|s|, r-1 -|s|\}$]\label{heur}
A heuristic argument for the bounded multiplier property \eqref{BMP}
under the conditions $-1+1/p<s<0<t<1/p$ was sketched in \cite[Remark 3.9]{Baladijoel}, exploiting
 via interpolation the fact  that
 (\cite[Thm 4.6.3/1]{RS}) the characteristic function of a half-plane in $\real^n$ is a bounded multiplier
on the Besov space $B^\tau_{p,\infty}(\real^n)$ if 
$
\frac{1}{p}-1 < \tau < \frac{1}{p}
$. 
It does not seem easy to fill in  details of
this argument, and we shall prove  Theorem~\ref{propp} using  paraproduct decompositions instead of interpolation.
The restriction $t-(r-1)<s<-t$ is in any case necessary for applications to hyperbolic dynamics,
and the bound for the essential spectral radius in \cite{Baladijoel} improves as
$p \to 1$.
\end{remark}

\subsection{Basic toolbox (Nikol'skij and Young bounds,  paraproduct decomposition, 
and a crucial trivial observation on functions of a single variable)}

The proofs below use the  {\it Nikol'skij inequality} 
  (see e.g.
\cite[Remark 2.2.3.4, p. 32]{RS}) which  says, in dimension $D\ge 1$, that
for any $p> p_1>0$ there exists $C$ so that for any $M>1$, and any $f$ with 
$\supp \, \FFF (f)\subset \{|\xi|\le M\}$, 
\begin{equation}\label{Nik}
\|f\|_{L_{p}(\real^D)}\le C M^{D (1/p_1-1/p)}\|f\|_{L_{p_1}(\real^D)}\, .
\end{equation}

We shall also use the  following {\it leafwise version of  Young's inequality}
(which can be proved like \cite[Lemma 4.2]{BT2}, see  \cite{Baladijoel}, by
using that any translation $\Gamma + x$ of $\Gamma\in \FF$ also belongs to  $\FF$):
\begin{equation}\label{magic5}
\| \tilde \psi * \varphi \|^s_{p, \Gamma}
\le  \|\tilde \psi\|_{L_1(\real^d)} \sup_{x\in \real^d} \| \varphi \|^s_{p,\Gamma+x}
\le \|\tilde \psi\|_{L_1} \sup_{\tilde \Gamma \in \FF} \| \varphi \|^s_{p,\tilde \Gamma}\, .
\end{equation}


Write $S_k \varphi=\psi_k^{Op} (\varphi)$ for $k\ge 0$,  set
$S_{-1} \varphi \equiv 0$, and  put
$
S^j \varphi=\sum_{k=0}^j S_k \varphi 
$ for  integer $j\ge 0$.
The (a priori formal) {\it paraproduct  decomposition} (see \cite[\S 4.4]{RS})  is 
\begin{align}
 \nonumber \varphi \cdot \upsilon &= \lim_{j\to \infty} (S^j \varphi) \cdot( S^j \upsilon)
 \\\nonumber &
 =
\sum_{k=2}^\infty \sum_{j=0}^{k-2} S_j \varphi \cdot S_k \upsilon
+\sum_{k=0}^\infty \sum_{j=k-1}^{k+1} S_j \varphi \cdot S_k \upsilon
+\sum_{j=2}^\infty \sum_{k=0}^{j-2} S_j \varphi \cdot S_k \upsilon \\
\label{para}&=\Pi_1(\varphi,\upsilon)+\Pi_2(\varphi,\upsilon)+\Pi_3(\varphi,\upsilon)\, ,
\end{align}
where we put
\begin{align*}
\Pi_1(\varphi ,\upsilon )&= 
\sum_{k=2}^\infty  S^{k-2}  \varphi \cdot S_k \upsilon\, ,\qquad \quad
\Pi_2(\varphi ,\upsilon )=
\sum_{k=0}^\infty (S_{k-1} \varphi + S_k \varphi+ S_{k+1}\varphi )  \cdot S_k \upsilon\, ,\\
&\qquad \mbox { and } \qquad \Pi_3(\varphi ,\upsilon )=\sum_{j=2}^\infty  S_j \varphi \cdot S^{j-2} \upsilon=\Pi_1(\upsilon,\varphi) \, .
\end{align*}
The two key facts motivating the  decomposition \eqref{para} are
\begin{equation}\label{f1}
\supp \, \FFF (S^{k-2} \varphi \cdot S_k \upsilon) \subset \{2^{k-3}\le  \|\xi \|\le 2^{k+1} \}\, ,
\quad \forall k \ge 2 \, , 
\end{equation}
and
\begin{equation}\label{f2}
\supp \,  \FFF \, (\sum_{j=k-1}^{k+1} S_j \varphi \cdot S_k \upsilon) \subset \{ \|\xi \|\le 5 \cdot 2^{k} \}\, ,
\quad \forall k \ge 0 \, .
\end{equation}

Finally, the proof of Theorem~\ref{propp} hinges on the fact that the singular set of a
characteristic function is co-dimension one:
We shall reduce there to the
case $\partial \tilde \Lambda=\{x_1=0\}$ so that 
$1_{\tilde \Lambda}$ only depends on the first coordinate $x_1$ of
$x\in \real ^d$. We shall use  below the fact that for such $\tilde \Lambda$
(see \cite[Lemma 4.6.3.2 (ii), p. 209, Lemma 2.3.1/3, p. 48]{RS})  for all $p \in (1, \infty)$
\begin{equation}\label{needit}
\| 1_{\tilde \Lambda}\|_{B^t_{p, q}(\real^d)} < \infty \, ,
\mbox{ if  } 0<t<1/p \mbox{ and } 0<q< \infty \mbox{ or } t=1/p \mbox{ and } q=\infty\, .
\end{equation}
We also
note for further use the {\it trivial but absolutely essential fact}
that if a function $\upsilon(x)$ only depends on
$x_1$  then $S_k \upsilon=(\FFF^{-1}\psi_k) * \upsilon$ also only
depends on $x_1$ for all $k$, and, more precisely,
\begin{equation}\label{keyo}
S_k \upsilon (x):= (\FFF^{-1} \psi_k) * \upsilon (x)=(\FFF^{-1} \psi^{(1)}_k) * \upsilon (x_1) \, .
\end{equation}
Indeed
$$
(\FFF^{-1} \psi_k) * \upsilon (x)=
\int  (\FFF^{-1} \psi_k) (y) \D y_2 \ldots \D y_d \, \upsilon (x_1-y_1)  \D y_1\, ,
$$
and, since $(2\pi)^{-(d-1)} \int_{\real^{d-1}} e^{\I (y_2,\ldots, y_d )(\xi_2,\ldots, \xi_d)}dy_2 \ldots dy_d$ (the inverse Fourier transform
of the constant function) is the
Dirac mass at $(\xi_2,\ldots,\xi_d)=0$, we get,
\begin{align*}
  \int_{\real^{d-1}} &(\FFF^{-1} \psi_k) (y_1, y_2, \ldots, y_d) \D y_2 \ldots \D y_d \\
&=
\frac{1}{(2\pi)^d} \int_{\real^{d-1}} 
  \int_{\real} \int_{\real^{d-1}}
e^{\I y_1 \xi_1} \psi_k(\xi) \D \xi_1 \D \xi_2 \ldots \D \xi_d
    e^{\I (y_2,\ldots ,y_d )(\xi_2,\ldots ,\xi_d)} \D y_2 \ldots \D y_d\\
 & = \frac{1}{2\pi}\int_{\real} e^{\I y_1 \xi_1} \psi_k(\xi_1,0) \D \xi_1 =(\FFF^{-1} \psi_k^{(1)})
(y_1) \, ,
\end{align*}
where we used that $\psi^{(d)}_k(\xi_1,0)=\psi^{(1)}_k(\xi_1)$.

\subsection{Multipliers depending on a single coordinate}

This subsection is devoted to a classical  property of multipliers depending 
on a single coordinate, which is instrumental in the
proof of Theorem~\ref{propp}.
If $1\le p\le \infty$, let $1\le p'\le \infty$ be so that 
\begin{equation}\label{p'}
\frac 1 p+\frac 1 {p'} =1\, ,\mbox{ i.e., } \, p'=\frac{p}{p-1}\, .
\end{equation}

\begin{lemma}\label{leboot}
Let $d_s\ge 1$.
Let $1<p<\infty$ and let
$-1+\frac{1}{p}<s< 0$. Then there exists $C<\infty$
so that for all  $f, g:\real^{d_s}\to \complex$ with 
 $g(x)=g(x_1)$, 
 \begin{equation}
\label{bootstrap}
\| f g\|_{B^s_{p,\infty}(\real^{d_s})}
\le C \| f \|_{B^s_{p,\infty}(\real^{d_s})}(\| g\|_{B^{1/ p'}_{ p',\infty}(\real)}+\|g\|_{L_\infty(\real)})\, .
\end{equation}
\end{lemma}

\begin{remark}The  bound \eqref{bootstrap} is a special case of a much more general
result (see e.g. \cite[Cor 4.6.2.1 (40)]{RS}) which also implies that
if $g(x)=g(x_1)$ then
\begin{equation}
\label{bootstrap'}
\| f g\|_{B^t_{p,\infty}(\real^{d_s})}
\le C \| f \|_{B^t_{p,\infty}(\real^{d_s})}(\limsup_{q\to p} \| g\|_{B^{1/ q}_{ q,\infty}(\real)}+\|g\|_{L_\infty(\real)})
\mbox{ if } 0<t< \frac{1}{p}\, ,
\end{equation}
for  a constant $C$, which may depend on $p$ and $t$,  
but not on $f$ or $g$.\end{remark}

For the convenience of the reader, and as a warmup in the use of
paraproducts, we include a proof of Lemma~\ref{leboot}.

\begin{proof}[Proof of Lemma~\ref{leboot}.]
The proof  uses the decomposition $\tilde \Pi_1(f,g)+\tilde \Pi_2(f,g)+\tilde \Pi_3(f,g)$  obtained 
from \eqref{para} by replacing 
$S_k$ and $S^k$ by the $d_s$-dimensional operators
\begin{equation}
\label{lowdim}
\tilde S_k:=(\psi_k^{(d_s)})^{Op}f\, , \qquad
\tilde S^k:=\sum_{j=0}^k  (\psi_j^{(d_s)})^{Op}f=
\sum_{j=0}^k \tilde S_j f\, .
\end{equation}
The bound for the contribution of $\tilde \Pi_3(f,g)$   is easy and does not require conditions
on $s$ or $g$: Indeed, \eqref{f1}  and the Young inequality   with 
the first claim of \eqref{forYoung} imply
\begin{align}
\nonumber &\|\sum_{j=2}^\infty
\tilde S_j f \tilde S^{j-2} g \|_{B^s_{p, \infty}(\real^{d_s})}
\le C \sup_{k\ge 2}  2^{ks}\sum_{\ell=-1}^{+3} \|\tilde S_{k+\ell} f \tilde S^{k+\ell-2} g \|_{L_p(\real^{d_s})}\, .
\end{align}
We focus on the term for $\ell=0$ (the others are similar) and get
\begin{align}
\sup_{k\ge 2}  2^{ks}\|\tilde S_k f \tilde S^{k-2} g \|_{L_p(\real^{d_s})}\label{Pi3} &
\le C \sup_k 2^{ks}\|\tilde S_k f  \|_{L_p(\real^{d_s})}\sup_k \|\tilde S^k g\|_{L_\infty}\\
\nonumber &\le C \|f\|_{B^s_{p, \infty}(\real^{d_s})}\| g\|_{L_\infty}\, ,
\end{align}
where we used the H\"older inequality and then the Young inequality, together with 
the second claim of \eqref{forYoung}. 

For $\tilde \Pi_1(f,g)$, we do not require any condition on $g$, and
the condition on $s$ is limited to $s<0$: Indeed,  exploiting again \eqref{f1}, we get
\begin{align}
\nonumber \|\sum_{j=2}^\infty
\tilde S^{j-2} f \tilde S_j g \|_{B^s_{p, \infty}(\real^{d_s})}
&\le C \sup_{k\ge 2}  2^{ks}\sum_{\ell=-1}^{+1}\|\tilde S^{k+\ell-2} f \tilde S_{k+\ell} g \|_{L_p(\real^{d_s})}\, .
\end{align}
Focusing again on the terms for $\ell=0$, we find
\begin{align}
\sup_{k\ge 2}  2^{ks}\|\tilde S^{k-2} f \tilde S_k g \|_{L_p(\real^{d_s})}
\nonumber &\le C \sup_k 2^{ks}\|\sum_{j=0}^{k-2} \tilde S_j f  \|_{L_p(\real^{d_s})}\sup_k \|\tilde S_k g\|_{L_\infty}\\
\nonumber &\le C \sup_k \bigl (\sum_{j=0}^{k-2}2^{(k-j)s} \bigr )\, \sup_j 2^{js}\| \tilde S_j f  \|_{L_p(\real^{d_s})} 
\| g\|_{L_\infty}\\
\label{Pi1} &\le C \|f\|_{B^s_{p, \infty}(\real^{d_s})}\| g\|_{L_\infty}\, , 
\end{align}
where we used the H\"older inequality and then the Young inequality, together with 
the first claim of \eqref{forYoung}.

\smallskip
The computation for $\tilde \Pi_2(f,g)$ is trickier and will use the assumption $s> -1+1/p$ together with
  the Nikol'skij inequality \eqref{Nik}.
For $\ell \in \{0, \pm 1\}$,  by \eqref{f2},
we get
\begin{align}\label{dontforget0}
\|\sum_{j=0}^\infty
\tilde S_{j+\ell} f \tilde S_j g \|_{B^s_{p, \infty}(\real^{d_s})}
\le 
C \sum_{j=0}^{\infty} \sup_{k\ge 0} 2^{ks}
  \|\tilde S_k(
\tilde S_{k+j+\ell} f \tilde S_{k+j} g) \|_{L_p(\real^{d_s})}\, .
\end{align}
In the sequel, we consider the terms with $\ell=0$ (the other terms are almost identical).
Setting $y=(x_2, \ldots, x_{d_s})$ and applying the one-dimensional Nikol'skij inequality 
\eqref{Nik}
for $1<p_1 < p$, we have, for any function $\upsilon$, 
\begin{align}
\label{!!}2^{ks}\|
\tilde S_{k}\upsilon \|_{L_p(\real^{d_s})}&= \biggl ( \int \biggl [\bigl ( \int 2^{k s p} |\tilde S_k \upsilon(x_1,y)|^{p}  \D x_1\bigr )^{1/p}\biggr ]^{p}
\D y\biggr) ^{1/p}\\
\nonumber &\le
\biggl ( \int \biggl [\bigl ( \int 2^{k (s+\frac{1}{p_1}-\frac{1}{p}) p_1} |\tilde S_k \upsilon (x_1,y)|^{p_1}  \D x_1\bigr )^{1/p_1}\biggr ]^{p}
\D y\biggr) ^{1/p}\\
\nonumber &=2^{k (s+\frac{1}{p_1}-\frac{1}{p})} A(p,p_1,\tilde S_{k}\upsilon)\, ,
\end{align}
where
\begin{equation}\label{defA}
A(p,p_1,\tilde S_{k}\upsilon)=\biggl ( \int \biggl [\bigl ( \int  |\tilde S_{k} \upsilon(x_1,y)|^{p_1}  \D x_1\bigr )^{1/p_1}\biggr ]^{p}
\D y\biggr) ^{1/p}\, .
\end{equation}
Since $s>- 1+1/p$, we may choose
$p_1\in (1,p)$ close enough to $1$ so that
\begin{equation}\label{defs}
s_1=s+ \frac{1}{p_1}-\frac{1}{p} >0 \, .
\end{equation}
Then, the right-hand side of \eqref{dontforget0} can be bounded as follows, using \eqref{!!},
\begin{align}\label{dontforget}
&\sum_{j=0}^{\infty}  
\sup_{k\ge 0} 
 2^{ks}\|\tilde S_k(
\tilde S_{k+j} f \tilde S_{k+j }g) \|_{L_p}
 \le \sum_{j=0}^{\infty}    \sup_k
 2^{k s_1} A(p,p_1,\tilde S_{k}(\tilde S_{k+j} f \tilde S_{k+j }g))\\
 \nonumber &\qquad\qquad\qquad \le   \bigl ( \sum_{j=0}^{\infty}  2^{-js_1}\bigr )
 \sup_{k,j}
 2^{(k+j) s_1} A(p,p_1,\tilde S_{k}(\tilde S_{k+j} f \tilde S_{k+j }g)) \\ 
\nonumber &\qquad \qquad\qquad\le C  \sup_{m\ge 0} 2^{ms_1} A(p,p_1,\tilde S_{m} f \tilde S_{m }g)\, .
\end{align}
In the last line we used \eqref{f2} to exploit  that there exists   $C<\infty$, depending on $p>1$ and $p_1>1$,  so that, for any  $\{\upsilon_k\}_{k \ge 0}$ so that 
$\supp \, (\FFF(\upsilon_k))\subset \{|\xi|\le 5 \cdot 2^{k}\}$,
\begin{align*}
A(p,p_1,\tilde S_{k}(\upsilon_{k+j}))&\le
C A(p,p_1,\upsilon_{k+j})\, ,\quad\forall
k\ge 0\, , \, j\ge 0 \, .
\end{align*}
(The above basically follows from
Young's inequality, see \cite[Thm 2.6.3, (5), p. 96]{RS}, noting that $p>1$ and $p_1>1$, so that
$\max\{0, 1/p-1, 1/p_1-1\}=0$, and noting that $f_j$ in the right-hand side of \cite[(5), p. 96]{RS}
should be replaced by $f_{j+\ell}$, see \cite[Thm 2.4.1.(II) and (III)]{Fr}.)

Next, recalling that $g$ only depends on $x_1$, using \eqref{keyo}, and applying
 the H\"older inequality in $dx_1$ for $1/p_1 =1/p+1/q$, we find $C$ so that for all $k$
\begin{align}
\nonumber A(p,p_1,\tilde S_{k} f \tilde S_{k}g)&=\biggl ( \int \biggl [\bigl ( \int  |\tilde S_k g(x_1)\tilde S_k f(x_1,y)|^{p_1}  \D x_1\bigr )^{1/p_1}\biggr ]^{p} \D y\biggr) ^{1/p}\\
\nonumber &\le C \biggl ( \int    \biggl [\bigl (\int | \tilde S_k g(x_1)|^{q} \D x_1\bigr )^{1/q}
 \bigl (\int | \tilde S_k f(x_1,y)|^{p}  \D x_1\bigr )^{1/p}\biggr ]^{p}
\D y\biggr) ^{1/p}\\
\nonumber &\le C
\bigl ( \int    | \tilde S_{k} g(x_1)|^{q} \D x_1 \bigr)^{1/q}
\biggl ( \int \biggl [\bigl ( \int  |\tilde S_k f(x_1,y)|^{p}  \D x_1\bigr )^{1/p}\biggr ]^{p}
\D y\biggr) ^{1/p}\\
\nonumber &= C \| \tilde S_{k} g\|_{L_q(\real)}\|\tilde S_k f\|_{L_p(\real^{d_s})} \, .
\end{align}
Note that \eqref{keyo} implies $\tilde S_{k} g=( \psi^{(1)}_k)^{Op}g$.
Finally, putting together \eqref{dontforget0} and \eqref{dontforget}, we find,
recalling \eqref{defs} and \eqref{p'},
\begin{align}
\nonumber \|\sum_{j=0}^\infty
\tilde S_{j} f \tilde S_j g \|_{B^s_{p, \infty}(\real^{d_s})}
&\le C  \sup_{k\ge 0} \bigl (2^{ks_1} \|\tilde S_{k} g\|_{L_q(\real)}\|\tilde S_k f\|_{L_p(\real^{d_s})}\bigr ) \\
\nonumber &\le C  \sup_{k\ge 0} \bigl ( 2^{k\frac{1}{q}} \|\tilde S_{k} g\|_{L_q(\real)}\bigr ) 
\sup_{k \ge 0} \bigl ( 2^{ks} \|\tilde S_k f\|_{L_p(\real^{d_s})}\bigr ) \\
\label{Nikk} &\le  C \sup_{k\ge 0}
\bigl (  2^{k\frac{1}{q}} 2^{k(\frac{1}{p'}-\frac{1}{q})} \| \tilde S_{k} g\|_{L_{p'}(\real)} 
\bigr ) \| f \|_{B^s_{p,\infty}(\real^{d_s})}\\
\label{Pi2} &\le  C  \| g\|_{B^{1/p'}_{p'}(\real)} \| f \|_{B^s_{p,\infty}(\real^{d_s})}\, ,
\end{align}
where we used the one-dimensional Nikol'skij inequality  for $q>p'>1$ in \eqref{Nikk}
(recalling \eqref{f2}).
Together, \eqref{Pi3}, \eqref{Pi1}, and \eqref{Pi2} give \eqref{bootstrap}.
\end{proof}


\subsection{Proof of Theorem~\ref{propp}}

To prove the theorem, we need one last lemma.
The point is that
if $\Gamma$ is horizontal, i.e., $\Gamma=\real^{d_s}\times \{0\}$, then \eqref{orthog} implies
\begin{equation}
\label{deco}
\tilde S_{k_s} ((S^{k} \varphi) \circ \pi_\Gamma^{-1}|_{\real^{d_s}}) \equiv 0
\, , \quad \forall k_s >  k+2  \ge 2 \, .
\end{equation}
If $\Gamma$ is an arbitrary admissible stable leaf, then we must work harder.
To state the bound replacing the trivial decoupling property \eqref{deco}, we
need notation: 
Defining  $b:\real^{d}\to 
\real_+$ by
$b(x)=1$ if $\|x\|\le 1$ and $b(x)=\|x\|^{-d-1}$ if $\|x\|> 1$,
we set 
$
 b_{k}(x)=2^{d k} \cdot b(2^{k} x)$ for $k\ge 0$.
(Note  that  $\|b_{k}\|_{L_1(\real^d)}=\|b\|_{L_1(\real^d)}<\infty$.)

\begin{lemma}[Decoupled wave packets in $\real^d$ and the cotangent space of
$\Gamma$]\label{sublemma4.4}
Fix a compact set  $K\subset \real^d$. There exists $C_0 \in [2, \infty)$
(depending on $\CC_\FF$, $K$) so that for any $k_s > k +C_0 \ge C_0$ and any 
$\Gamma \in \FF$,
the
kernel $V(x,y)$ defined by $\tilde S_{k_s} ((S^{k}  \varphi) \circ \pi_\Gamma^{-1}) (x)=\int_{y\in  \real^d} V(x,y) \varphi(y) \D y$ for $x\in \real^{d_s}$
and   $\varphi$ supported in $\Omega$  satisfies\footnote{The proof shows that the same bound holds for the kernel associated to
$\tilde S_{k_s} ((S_{k}  \varphi) \circ \pi_\Gamma^{-1}) (x)$.}
\begin{equation}\label{convol}
| V(x,y)|\le C_0 2^{-k_s r} 
    b_{k}(\pi_\Gamma^{-1} (x)-y) \, ,
\quad \forall x \in \real^{d_s} \, , \forall y \in \real^{d}\, .
\end{equation}
\end{lemma}

The lemma implies that $\int_{y\in  \real^d} V(x,y) \varphi(y) \D y$
is bounded by a convolution with a function in $L_1(\real^d)$, for which  \eqref{magic5} holds.

\begin{proof}
The kernel $V(x,y)$ is given by the formula\footnote{Strictly speaking, we must first integrate by parts $d_s+1$
times in the kernel $\int \E^{\I (\pi_\Gamma^{-1} (z)- y)\eta} 
\sum_{j=0}^k \psi_{j}(\eta) \D \eta$ of $(S^k\varphi)\circ \Pi_\Gamma^{-1}(z)$ for
$d(z,K)> \epsilon$, to get an element of $L_1(dz)$.}
\begin{align*}
\frac{1}{(2\pi)^{d_s+d}}
 \int_{z\in \real^{d_s}}\int_{\eta \in \real^{d}} \int_{\eta_s \in \real^{d_s}}
\E^{\I (\pi_\Gamma^{-1} (z)- y)\eta} \E^{\I (x- z)\eta_s} 
\sum_{j=0}^k \psi_{j}(\eta)
\psi_{k_s}^{(d_s)}(\eta_s) \D \eta_s \D \eta \D z\, .
\end{align*}

As a warmup, let us prove \eqref{deco} if $\Gamma$  is horizontal or,
more generally, affine: Letting 
 $\eta=(\eta_-,\eta_+)$ with $\eta_-=\pi_-(\eta)\in \real^{d_s}$, 
we have $\pi_\Gamma^{-1} (z)=(z,A(z)+A_0)$ with $A_0\in \real^{d_u}$
and $A:\real^{d_s}\to \real^{d_u}$ 
linear ($A\equiv 0$ if $\Gamma$ is horizontal), so that (using like in
\eqref{keyo} that $\FFF^{-1}(1)$ is the
Dirac at $0$), $V(x,y)$ can be rewritten as
\begin{align*}
& \frac{1}{(2\pi)^{d+d_s}}
 \int_{\real^{2d_s+d}}\E^{-\I  y\eta}
\E^{\I  x\eta_s} \E^{\I  A_0 \eta_+}
 \E^{\I z (-\eta_s+\eta_- + A^{tr}\eta_+)}  
\sum_{j=0}^k \psi_{j}(\eta)
\psi_{k_s}^{(d_s)}(\eta_s) \D \eta_s \D \eta \D z\\
&=
\frac{1}{(2\pi)^{d}}
 \int_{\real^{d}}\E^{-\I  y\eta}
\E^{\I  x(\eta_-+A^{tr}\eta_+)} \E^{\I  A_0 \eta_+}
\sum_{j=0}^k \psi_{j}(\eta)
\psi_{k_s}^{(d_s)}(\eta_-+A^{tr}\eta_+)  \D \eta \equiv 0
\, ,
\end{align*}
since $\psi_{j}(\eta)$ and
$\psi_{k_s}^{(d_s)}(\eta_-+A^{tr}\eta_+)$ have disjoint supports
if $k_s>k+C_0$, where $C_0\ge 2$ depends on $\|A\|\le \CC_\FF$.

More generally,  $\Gamma\in \FF$ is the graph of a $C^r$ map $\gamma$ (with $\|\gamma\|_{C^r}\le \CC_\FF$), i.e.,
$\pi_\Gamma^{-1}(z)=(z,\gamma(z))$ for $z \in \real^{d_s}$. 
The lemma is thus obtained
integrating by parts $r$ times 
(in the sense of \cite[App.~C]{Baladijoel} if $r$ is not an integer)
with respect to $z$ in the kernel $V(x,y)$, using \eqref{betagood},
and proceeding as in the end of the proof of \cite[Lemma 2.34]{Ba}, mutatis mutandis
(using that $\|y-\pi_\Gamma^{-1}(x)\|>2^{-k}$ 
implies that either $\|y-\pi_\Gamma^{-1}(z)\|>2^{-k+1}$
or $\|\pi_\Gamma^{-1}(z)-\pi_\Gamma^{-1}(x)\|>2^{-k+1}$, choosing $C_0$
depending on $\CC_\FF$,
so  that  $\|\pi_\Gamma^{-1}(z)-\pi_\Gamma^{-1}(x)\|>2^{-k+1}$ 
implies $\|z-x\|\ge 2^{-k+1}/C_0$).
\end{proof}

\begin{proof}[Proof of Theorem~\ref{propp}]
If $G$ is a rotation about $0\in \real ^d$  then,  since
$\psi_n \circ G^{-1}=\psi_n$, we have
 $\psi_n^{Op}(\tilde  \varphi \circ G)= ((\psi_n\circ G^{tr})^{Op}\tilde \varphi)\circ G=(\psi_n^{Op}\tilde \varphi)\circ G$
 (use $G^{tr}=G^{-1}$),
and thus
  $\| \tilde \varphi \circ G\|_{\UU_p^{\cone_+,t,s}} =
\|\tilde \varphi \|_{\UU_p^{G(\cone_+),t,s}}$ for all $\varphi$ (use $G \circ \pi_\Gamma^{-1}= \pi_{G(\Gamma)}^{-1}$).
It thus suffices to show \eqref{BMP} for 
$\Lambda =\{ x \in \real^d \mid x_1>0 \}$. 
Indeed, the assumption on $u_{\tilde \Lambda}$  implies that  
the  rotation $G$ satisfying $1_{\tilde \Lambda}  \varphi =(1_{\Lambda} (\varphi \circ G^{-1})) \circ G$
is such that $G(\cone_+)$ is still an unstable cone, i.e., $\real^{d_s} \times \{0\}$ is included in $(\real^d \setminus G(\cone_+)) \cup \{0\}$ (note that $G(u_{\tilde \Lambda})=(1,0, \ldots ,0)$, and consider the
limiting case $u_{\tilde \Lambda}\to \partial \cone_+$).

Next, since $\varphi$ is supported in $K$, 
we can replace the half-space $\Lambda$
by a strip $0<x_1<B$, still denoted $\Lambda$, and whose characteristic function  $1_\Lambda(x)$   still only depends on $x_1\in \real$. Without loss of generality, we may assume that $B=1$. 

Our starting point is then the decomposition \eqref{para} applied to 
$\upsilon=1_\Lambda$. 
We consider first the term $\Pi_3(\varphi ,1_\Lambda )$. We will bootstrap
from  Lemma~\ref{leboot}. Set
\begin{equation}
\label{gGamma}
1_\Lambda^{k-2,\Gamma}(x_-)= (S^{k-2} 1_\Lambda)(x_-, \gamma(x_-))
=\sum_{j=0}^{k-2} (\FFF^{-1} \psi_j *1_\Lambda) (x_-,\gamma(x_-))\, .
\end{equation}
Then $1_\Lambda^{k-2,\Gamma}(x_-)$ is a function of $x_1$ alone (recalling \eqref{keyo}), and the leafwise Young inequality \eqref{magic5}, 
together with 
the second claim of \eqref{forYoung} 
 and the fact that $\| 1_\Lambda \|_{B^{1/t}_{t,\infty}(\real)}<\infty$
(for  any $1<t<\infty$, see e.g. \cite[Lemma 2.3.1/3(ii), Lemma 2.3.5]{RS}), 
give that both $\| 1_\Lambda^{k-2,\Gamma}\|_{B^{1/ p'}_{p',\infty}(\real)}$
and $\|1_\Lambda^{k-2,\Gamma}\|_{L_\infty(\real)}$ are finite, uniformly in $\Gamma$ and $k$.
Next, by \eqref{f1}, \eqref{magic5}, and \eqref{bootstrap}, there exists
a constant $C$ so that for  any $\ell \ge 0$, since $-1+1/p<s < 0$,
\begin{align*}
2^{\ell t} \| S_\ell(\Pi_3(\varphi ,1_\Lambda ))\|^s_{p,\Gamma}&\le
2^{\ell t} \sum_{k=\ell-1}^{\ell+3} \|S_k \varphi \cdot S^{k-2} 1_\Lambda \|^s_{p,\Gamma}\\
&\le
2^{\ell t} \sum_{k=\ell-1}^{\ell+3} \|S_k \varphi \|^s_{p,\Gamma} 
(\| 1_\Lambda^{k-2,\Gamma}\|_{B^{1/p'}_{ p',\infty}(\real)}+\|1_\Lambda^{k-2,\Gamma}\|_{L_\infty(\real)}) \\
&\le C \sup_n
2^{n t}  \|S_n \varphi \|^s_{p,\Gamma}\le  C \|\varphi \|_{\UU^{\cone, t,s}_p}\, , 
\end{align*}
where we  used \eqref{bootstrap} from Lemma~\ref{leboot}
for $f(x_-)= S_k \varphi (x_-,\gamma(x_-))$
with $\gamma=\gamma(\Gamma)$ from the proof of Sublemma~\ref{sublemma4.4},
and $g(x_-)=1_\Lambda^{k-2,\Gamma}(x_-)$.  This concludes the bound for
$\Pi_3(\varphi, 1_\Lambda)$, and
we move to $\Pi_2(\varphi,1_\Lambda)$. Setting
\begin{equation}
\label{gGamma'}
1_{\Lambda,k}^\Gamma(x_-)= (S_k 1_\Lambda)(x_-, \gamma(x_-))
=(\FFF^{-1} \psi_k *1_\Lambda) (x_-,\gamma(x_-))\, ,
\end{equation}
we have that $1_{\Lambda,k}^\Gamma(x_-)=1_{\Lambda,k}^\Gamma(x_1)$, and also,  recalling
\eqref{needit}, the 
leafwise Young inequality
\eqref{magic5},   together with the first claim of \eqref{forYoung},
we find
\begin{equation}\label{arefinite}
\sup_{k,\Gamma}\| 1_{\Lambda,k}^\Gamma\|_{B^{1/p'}_{p',\infty}(\real)}<\infty
\, ,\qquad
\sup_{k,\Gamma}\| 1_{\Lambda,k}^\Gamma\|_{L_\infty(\real)}<\infty \, .
\end{equation}
Thus, using \eqref{f2}, and applying \eqref{bootstrap}
from Lemma~\ref{leboot} again, we find, since $t>0$,
\begin{align*}
2^{\ell t}& \| S_\ell(\Pi_2(\varphi ,1_\Lambda ))\|^s_{p,\Gamma}\le
2^{\ell t} 3\sum_{k\ge \ell-1} \|S_k \varphi \cdot S_{k} 1_\Lambda \|^s_{p,\Gamma}\\
&\qquad \le 3 \sup_k
2^{k t}\|S_k \varphi \|^s_{p,\Gamma} 
(\| 1_{\Lambda,k}^\Gamma\|_{B^{1/p'}_{ p',\infty}(\real)}+\|1_{\Lambda,k}^\Gamma\|_{L_\infty(\real)})
 \sum_{k\ge \ell-1}  2^{(\ell-k)t} \\
&\qquad \le C \sup_k
2^{k t}  \|S_k \varphi \|^s_{p,\Gamma}\le  C \|\varphi \|_{\UU^{\cone_+,t,s}_p}\, , \qquad
\forall \ell \ge 0 \, . 
\end{align*}

It remains to bound the contribution of $\Pi_1(\varphi,1_\Lambda)$. This is the trickiest
estimate. It will use Lemma~\ref{sublemma4.4} and our assumption 
$t-(r-1)<s<-t<0$.
For  any $\ell \ge 0$, we have, using again \eqref{magic5}, \eqref{f1}, and \eqref{forYoung},
\begin{align}\label{cliffh}
2^{\ell t}& \| \psi_\ell^{Op}(\Pi_1(\varphi ,1_\Lambda ))\|^s_{p,\Gamma}\le
 \sum_{k=\ell-1}^{\ell+3}2^{\ell t}  \|S^{k-2} \varphi \cdot S_k 1_\Lambda \|^s_{p,\Gamma}\, .
\end{align}
We may focus on the term $k=\ell$, as the others are almost identical. We will use
the paraproduct decomposition  $\tilde \Pi_1+\tilde \Pi_2+\tilde \Pi_3$ and the operators
$\tilde S_j$ and $\tilde S^j$
(see \eqref{lowdim}). Put $(S^{k-2} \varphi)^\Gamma=(S^{k-2} \varphi)\circ \pi_\Gamma^{-1}$.
 By \eqref{keyo} and \eqref{f1}, we have
\begin{align}
 \label{triv?}2^{kt}\|S^{k-2} \varphi \cdot S_k 1_\Lambda \|^s_{p,\Gamma}&\le
 \sum_{i=1}^2  2^{kt} \|\tilde \Pi_i((S^{k-2} \varphi)^ \Gamma , 1_{\Lambda,k}^\Gamma )\|_{B^s_{p,\infty}} 
 + 2^{kt} \RR^\Gamma_{k,s, p,\Lambda}(\varphi)
 \\
 \label{triv2} &\qquad +2^{kt} \sum_{m=k-1}^{k+1}\sum_{j=m+2}^{m+2+C_0} \|\tilde S_j((S^{k-2} \varphi)^\Gamma) (\tilde S_m  1^\Gamma_{\Lambda,k})\|_{
 B^s_{p,\infty}}\, ,
\end{align}
taking $C_0\ge 2$ from Lemma~\ref{sublemma4.4}, using \eqref{orthog}, and setting
$$
\RR^\Gamma_{k,s, p,\Lambda}(\varphi)=\sum_{m=k-1}^{k+1}\,\,
\sum_{j= m+2+C_0+1}^\infty \|\tilde S_j((S^{k-2} \varphi)^\Gamma) (\tilde S_m  1^\Gamma_{\Lambda,k})\|_{
 B^s_{p,\infty}} \, .
$$
Lemma~\ref{leboot} and the Young inequality (thrice) give $C$ so that for all $j, k, m,$ and $\Gamma$
\begin{align}
\nonumber &\|\tilde S_j((S^{k-2} \varphi)^\Gamma) (\tilde S_m  1^\Gamma_{\Lambda,k})\|_{
 B^s_{p,\infty}(\real^{d_s})}\\
\nonumber &\qquad\qquad\qquad\le C \| \tilde S_j((S^{k-2} \varphi)^\Gamma)\|_{B^s_{p,\infty}(\real^{d_s})}
 (\|  1^\Gamma_{\Lambda,k}\|_{B^{1/p'}_{p',\infty}(\real)}
 +\| 1^\Gamma_{\Lambda,k}\|_{L_{\infty}(\real)})
 \\
\label{rem} &\qquad\qquad\qquad\le C \| \tilde S_j((S^{k-2} \varphi)^\Gamma)\|_{B^s_{p,\infty}(\real^{d_s})}
 \le  C \| (S^{k-2} \varphi)^\Gamma\|_{B^s_{p,\infty}(\real^{d_s})}
 \, ,
\end{align}
where we applied \eqref{arefinite} in the second inequality.
Thus, Lemma~\ref{sublemma4.4} 
and the leafwise\footnote{See \S4 of Corrections and complements to \cite {Baladijoel} for the factor $2^{k(-s+\delta)}$.} Young inequality  \eqref{magic5}  
applied to $k_s=j \ge k+2+C_0$ gives $k_0\ge C_0$ so that for any $\delta\in (0,1)$ (recalling $0<t-s<r-1<r-\delta$)
\begin{align}
\nonumber
\sup_{k\ge k_0,\Gamma} 2^{kt} \RR^\Gamma_{k,s,p,\Lambda}(\varphi)
&\le 3 C_0 C \sup_{k,\Gamma} 2^{k(t-r-s+\delta)} \bigl (\sum_{j=k+2+C_0}^\infty 2^{-(j-k)r} \bigr )\, 
\|S^{k -2}\varphi \|^s_{p,\Gamma} \\
\label{wow} & \le 3 C_0 C \|\varphi\|_{\UU^{\cone_+,t,s}_p}\, .
\end{align}
Using again \eqref{rem}, the finite double sum in \eqref{triv2} is  bounded by  $(C_0+4) C \|\varphi\|_{\UU^{\cone_+,t,s}_p}$.

For the contribution of
$\tilde \Pi_1$ in \eqref{triv?}, using again \eqref{keyo} and \eqref{f1}, we find 
\begin{align*}
  2^{k t} & \|\tilde \Pi_1((S^{k-2} \varphi )^\Gamma,  1_{\Lambda,k}^\Gamma) \|_{B^s_{p,\infty}(\real^{d_s})}\le 2^{k t} \sum_{n=k-1}^{k+1}
  \|(\tilde S^{n-2} (S^{k-2} \varphi)^\Gamma) \cdot \tilde S_n(1_{\Lambda,k}^\Gamma)) \|_{B^s_{p,\infty}}
  \, .
  \end{align*}
Setting $(S_j \varphi)^\Gamma=(S_j \varphi)\circ \pi_\Gamma^{-1}$,
we bound the term for $n=k$ above\footnote{The other terms are similar.}
 by the sum  of
$$
2^{kt}  \sum_{\ell=-1}^1 2^{(k+\ell)s} \| 
[\sum_{j=0}^{k-2}\sum_{m=j+C_0}^{k-2} \tilde S_{k+\ell} (\tilde S_m  (S_j  \varphi)^\Gamma) ]\cdot
\tilde S_k(1_{\Lambda,k}^\Gamma)\|_{L_{p}(\real^{d_s})}\, ,
$$
(which
can be handled as in  \eqref{wow}, by Lemma~\ref{sublemma4.4}),
and, 
\begin{align*}
&2^{kt}  \sum_{\ell=-1}^1 2^{(k+\ell)s}  \| 
[\sum_{j=0}^{k-2}\, \,\sum_{m=0}^{j+C_0-1} \tilde S_{k+\ell} (\tilde S_m  (S_j  \varphi)^\Gamma) ]\cdot
\tilde S_k(1_{\Lambda,k}^\Gamma)\|_{L_{p}(\real^{d_s})}\\
& \le  (\sup_{0\le j\le k-2} \sum_{m=0}^{j+C_0-1}
2^{(j-m)s} )\\
&\qquad\quad \cdot
2^{kt} \sum_{\ell=-1}^1 \sum_{j=0}^{k-2}
\sup_{0\le m< j+C_0} 2^{(k+\ell-j+m)s} 
\| [\tilde S_{k+\ell}(\tilde S_m  (S_j  \varphi)^\Gamma) ]\cdot \tilde S_k(1_{\Lambda,k}^\Gamma)\|_{L_{p}}\\
&\le C 2^{C_0|s|} \sum_{j=0}^{k-2}
\sup_{\substack{0\le m< j+C_0\\-1\le \ell \le 1}} 2^{(k+\ell-j)(t+s)}  2^{ms} 2^{jt} 
\| \tilde S_{k+\ell} ([\tilde S_m  (S_j  \varphi)^\Gamma ]\cdot \tilde S_k(1_{\Lambda,k}^\Gamma))\|_{L_p(\real^{d_s})}\, , 
\end{align*}
using that $s<0$.
Now, since $s+t<0$, we get, using the Young inequality,
\begin{align*}
&\sum_{j=0}^{k-2}
\sup_ {\substack{0\le m< j+C_0\\-1\le \ell \le 1}}2^{(k+\ell-j)(t+s)}  2^{ms} 2^{jt} 
\| \tilde S_{k+\ell} ([\tilde S_m  (S_j  \varphi)^\Gamma ]\cdot \tilde S_k(1_{\Lambda,k}^\Gamma))\|_{L_p(\real^{d_s})}\\
&\qquad \le C \sup_m \sup_j 
 2^{ms} 2^{jt} 
\| \tilde S_m  (S_j  \varphi)^\Gamma \|_{L_p(\real^{d_s})}
\|  \tilde S_k(1_{\Lambda,k}^\Gamma)\|_{L_\infty(\real)}\\
&\qquad \le C \sup_j 2^{jt} \|S_j \varphi\|^s_{p,\Gamma}
\le C \|\varphi\|_{\UU^{\cone_+,t,s}_p}\, .
\end{align*}
Finally, using \eqref{keyo} once more, we bound the contribution of $\tilde \Pi_2$ in \eqref{triv?}:
\begin{align}
\nonumber  &2^{k t}  \|\tilde \Pi_2((S^{k-2} \varphi )^\Gamma,  1_{\Lambda,k}^\Gamma )\|_{B^s_{p,\infty}}\le
2^{kt}
\sum_{\ell=-1}^1 \|(\tilde S_{k+\ell}(S^{k-2} \varphi )^\Gamma) \cdot  \tilde S_k(1_{\Lambda,k}^\Gamma) \|_{B^s_{p,\infty}} \\
\label{dsum}&\le 2^{kt} \tilde \RR^\Gamma_{k,p,s,\Lambda}(\varphi)+
2^{kt}
\sum_{\ell=-1}^1\sum_{\tilde \ell=2}^{C_0}
\|(\tilde S_{k+\ell}(S_{k-\tilde\ell}  \varphi )^\Gamma)  \cdot \tilde S_k(1_{\Lambda,k}^\Gamma) \|_{B^s_{p,\infty}(\real^{d_s})}
\, ,
\end{align}
where 
$$2^{kt} \tilde \RR^\Gamma_{k,p,s,\Lambda}(\varphi)
=
2^{kt}
\sum_{\ell=-1}^1\sum_{\tilde \ell=C_0+1}^{k}
\|(\tilde S_{k+\ell}(S_{k-\tilde\ell}  \varphi )^\Gamma)  \cdot \tilde S_k(1_{\Lambda,k}^\Gamma) \|_{B^s_{p,\infty}(\real^{d_s})}
$$ 
can be bounded similarly as \eqref{wow}, using
Lemma~\ref{sublemma4.4}. For the remaining finite double sum in \eqref{dsum},
we focus on the contributions with $\ell=0$ and  $\tilde \ell=2$, the others being similar. Then, applying
Lemma~\ref{leboot}, we find
\begin{align*}
&\sup_{k, \Gamma} 2^{kt} \|(\tilde S_{k}(S_{k-2} \varphi )^\Gamma)  \cdot \tilde S_k(1_{\Lambda,k}^\Gamma) \|_{B^s_{p,\infty}(\real^{d_s})}
 \\
&\,\, \le \sup_{k, \Gamma} 2^{kt}
\|(\tilde S_{k}(S_{k-2} \varphi )^\Gamma) \|_{B^s_{p,\infty}(\real^{d_s})}
(\| 1_{\Lambda,k}^\Gamma \|_{B^{1/p'}_{p',\infty}(\real)} +\| 1_{\Lambda,k}^\Gamma \|_{L_{\infty}(\real)})\le C \|\varphi\|_{\UU^{\cone_+,t,s}_p}
\, ,
\end{align*} 
using  \eqref{arefinite} once more. 
This ends the proof of Theorem~\ref{propp}.
\end{proof}


\end{document}